\documentclass[10pt,reqno]{amsart}
\usepackage{bbm}
\usepackage{dutchcal}
\usepackage{mathrsfs}
\usepackage{diagbox}
\usepackage{cite}
\usepackage{amsfonts} 
\usepackage[dvipsnames,usenames]{color}
\usepackage{graphicx,latexsym,bm,amsmath,amssymb,verbatim,multicol,lscape}
\makeatletter
\@namedef{subjclassname@2020}{%
  \textup{2020} Mathematics Subject Classification}
\makeatother
\newtheorem{theorem}{Theorem} [section]

\newtheorem{conjecture}[theorem]{Conjecture}
\newtheorem{lemma}[theorem]{Lemma}

\numberwithin{equation}{section}
\allowdisplaybreaks

\begin{document}
\title{Divisibility among power GCD and power LCM matrices
on certain gcd-closed sets}
\begin{abstract}
Let $(x, y)$ and $[x, y]$ denote the greatest common divisor
and the least common multiple of the integers $x$ and $y$
respectively. We denote by $|T|$ the number
of elements of a finite set $T$. Let $a,b$ and $n$ be
positive integers and let $S=\{x_1, ..., x_n\}$ be a set of $n$
distinct positive integers. We denote by $(S^a)$ (resp. $[S^a]$)
the $n\times n$ matrix whose $(i,j)$-entry is the $a$th power of $(x_i,x_j)$
(resp. $[x_i,x_j]$). For any $x\in S$, define
$G_{S}(x):=\{d\in S: d<x, d|x \ {\rm and} \
(d|y|x, y\in S)\Rightarrow y\in \{d,x\}\}$.
In this paper, we show that if $a|b$ and $S$ is
gcd closed (namely, $(x_i, x_j)\in S$ for all integers
$i$ and $j$ with $1\le i, j\le n$) and
$\max_{x\in S}\{|G_S (x)|\}=2$
and the condition $\mathcal{G}$ being satisfied (i.e.,
any element $x\in S$ satisfies that either
$|G_S(x)|\le 1$, or $G_S(x)=\{y_1,y_2\}$ satisfying that
$[y_1,y_2]=x$ and $(y_1,y_2)\in G_S(y_1)\cap G_S(y_2)$),
then $(S^a)|(S^b), (S^a)|[S^b]$ and $[S^a]|[S^b]$
hold in the ring $M_{n}({\bf Z})$. Furthermore, we show the
existences of gcd-closed sets $S$ such that $S$ does not
satisfy the condition $\mathcal{G}$ and such factorizations
are true. Our result extends the Feng-Hong-Zhao theorem
gotten in 2009. This also partially confirms a conjecture raised by Hong
in [S.F. Hong, Divisibility among power GCD matrices and power LCM matrices,
{\it Bull. Aust. Math. Soc.}, doi:10.1017/S0004972725100361].
\end{abstract}
\author[J.X. Wan]{Jixiang Wan}
\address{College of Mathematics and Physics, Mianyang Teachers' College,
Mianyang, 621000, P.R. China}
\email{xiangxiangwanli@163.com; 2873555@qq.com}
\author[G.Y. Zhu]{Guangyan Zhu$^*$}
\address{School of Mathematics and Statistics, Hubei Minzu University,
Enshi 445000, P.R. China}
\email{2009043@hbmzu.edu.cn}
\thanks{$^*$ G.Y. Zhu is the corresponding author. 
The research was supported in part by the Startup
Research Fund of Hubei Minzu University for Doctoral Scholars
(Grant No. BS25008)}
\maketitle

\section{Introduction}
Let ${\bf Z}$ and ${\bf Z}^+$ denote the ring of integers
and the set of positive integers. Let $|T|$
stand for the cardinality of a finite set $T$ of integers.
For any $x, y\in{\bf Z}^+$, let $(x, y)$ and $[x, y]$
denote their greatest common divisor and least common
multiple, respectively. Let $f$ be an arithmetic function
and let $S=\{x_1, ..., x_n\}$
be a set of $n(\in {\bf Z}^+)$ distinct positive integers.
Let $(f(S))$ and $(f[S])$ denote the $n\times n$ matrices
having $f((x_i, x_j))$ and $(f[x_i,x_j])$ as its $(i,j)$-entry,
respectively. We say that $S$ is {\it factor closed} (FC)
if $[x\in S, d|x, d>0]\Rightarrow [d\in S]$, and that
$S$ is {\it gcd closed} if $(x_i, x_j)\in S \ \forall
1\le i,j\le n$. Then any FC set is gcd closed but not
conversely. Smith \cite{S-PLMS75} proved that
$\det(f(x_i, x_j))=\prod_{k=1}^n(f*\mu)(x_k)$,
where $S$ is FC and $(f*\mu)(x):=\sum_{d|x}f(d)\mu
\big(\frac{x}{d}\big)$ for any positive integer $x$.
Since then lots of generalizations of Smith's theorem
and related results have published (see, for example,
\cite{AYK-LAA17}-\cite{[Mc]} and \cite{T}-\cite{ZLX}).
The function $\xi_a$ is defined for any positive
integer $x$ by $\xi_a(x):=x^a$. The $n\times n$ matrix
$(\xi_a(x_i, x_j))$ (abbreviated by $(S^a)$) and
$(\xi_a[x_i, x_j])$ (abbreviated by
$[S^a]$) are called {\it $a$th power GCD matrix} on
$S$ and {\it $a$th power LCM matrix} on $S$, respectively.
In 1993, Bourque and Ligh \cite{BL-LMA93} extended the
Beslin-Ligh result \cite{BL-BAMS89} and Smith's
determinant by proving that if $S$ is gcd closed,
then $\det(S^a)=\prod_{k=1} ^n \alpha_{a, k}$, where
\begin{align}\label{eq1.2}
\alpha_{a, k}:=\sum_{d|x_k\atop d\nmid x_t,x_t<x_k}
(\xi_a*\mu)(d).
\end{align}

For any $x,y \in S$ with $x<y$, we say that $x$
is a {\it greatest-type divisor} of $y$ in $S$
if $[x|y,x|d|y,d\in S]\Rightarrow [d\in \{x,y\}]$.
For $x\in S$, $G_S(x)$ denotes the set of all
greatest-type divisors of $x$ in $S$. The concept
of greatest-type divisor was introduced by Hong and
played a key role in his solution \cite{H-JA99} of
the Bourque-Ligh conjecture \cite{BL-LAA92}. Bourque
and Ligh \cite{BL-LAA95} showed that
if $S$ is FC and $a\ge 1$ is an integer, then the $a$th
power GCD matrix $(S^a)$ divides the $a$th power LCM matrix
$[S^a]$ in the ring $M_n({\bf Z})$ of $n\times n$ matrices
over the integers. That is, $\exists A\in M_n({\bf Z})$
such that $[S^a]=(S^a)A$ or $[S^a]=A(S^a)$. Hong \cite{H-LAA02}
showed that such a factorization is no longer true in general
if $S$ is gcd closed and $\max_{x \in S}\{|{G_S(x)}|\}=2$.
These results were extended by Korkee and Haukkanen
\cite{KH-LAA08} and Chen et al. \cite{CFHQ-PMD22}.
Hong \cite{H-LAA08} is the first one investigating the
divisibility properties among power GCD matrices and
among power LCM matrices. In fact, he showed that
$(S^a)|(S^b), (S^a)|[S^b]$ and $[S^a]|[S^b]$ hold
in $M_{n}({\bf Z})$ if $a|b$ and $S$ is a divisor
chain (that is, $x_{\sigma(1)}|...|x_{\sigma(n)}$
for a permutation $\sigma$ of $\{1,..., n\}$), and
such factorizations are no longer true if $a\nmid b$
and $|S|\ge 2$. Evidently, a divisor chain is gcd
closed but not conversely. In 2022 and 2023, Zhu
\cite{Z-IJNT22} and Zhu and Li \cite{ZL-BAMS22}
confirmed three conjectures of Hong \cite{H-LAA08}
by showing that $(S^a)|(S^b), (S^a)|[S^b]$ and
$[S^a]|[S^b]$ hold in $M_{n}({\bf Z})$ when
$a|b$ and $S$ is a gcd-closed set with
$\max_{x\in S}\{|G_S (x)|\}=1$. In 2022, Zhu,
Li and Xu \cite{ZLX} showed the existences of
gcd-closed sets $S$ with $\max_{x\in S}
\{|G_{S}(x)|\}=2$ and infinitely many integers
$b\ge 2$ such that $(S)\nmid (S^b)$ (resp.
$(S)\nmid [S^b]$ and $[S]\nmid [S^b]$).
As shown in \cite{[H-BAMS25]}, for any set
$S$ of positive integers and for
any $x\in S$ with $|G_S(x)|\ge 2$, we say that
the two distinct greatest-type divisors $y_1$
and $y_2$ of $x$ in $S$ {\it satisfy the condition
$\mathcal{G}$} if $[y_1,y_2]=x$ and $(y_1,y_2)\in G_S(y_1)\cap G_S(y_2)$.
We say that $x$ {\it satisfies the condition $\mathcal{G}$} if any two
distinct greatest-type divisors of $x$ in $S$ satisfy the condition $\mathcal{G}$.
Moreover, we say that a set $S$ of positive integers {\it satisfies the
condition $\mathcal{G}$} if any element $x$ in $S$ satisfies that either
$|G_S(x)|\le 1$, or $|G_S(x)|\ge 2$ and $x$ satisfies the condition $\mathcal{G}$.
The following conjecture was proposed in
the last section of \cite{[H-BAMS25]}.

\begin{conjecture}{\rm {\cite[Conjecture 3.4]{[H-BAMS25]}}} \label{p1.1}
Let $a$ and $b$ be positive integers with $a|b$ and let $S$ be
a gcd-closed set satisfying the condition $\mathcal G$. Then
$(S^a)|(S^b)$, $(S^a)|[S^b]$ and $[S^a]|[S^b]$ in the ring $M_{|S|}(\mathbb Z)$.
\end{conjecture}

By Zhu's theorem \cite{Z-IJNT22} and the Zhu-Li result \cite{ZL-BAMS22},
one knows that this conjecture is true when $\max_{x\in S}\{|G_S (x)|\}=1$.
For $\max_{x\in S}\{|G_S (x)|\}=2$, notice that the condition $\mathcal G$
is the condition $\mathcal C$ in \cite{FHZ-DM09}.
When $a=b$ and $\max_{x\in S}\{|G_S(x)|\}=2$, Feng, Hong and Zhao
\cite{FHZ-DM09} verified this conjecture. For the case when $a|b$ and $a<b$,
and $\max_{x\in S}\{|G_S (x)|\}\ge 2$, Conjecture \ref{p1.1} remains widely open.
In this paper, our main goal is to explore Conjecture \ref{p1.1} for the case
$\max_{x\in S} \{|G_S(x)|\}=2$. The main
results of this paper can be stated as follows.

\begin{theorem}\label{t1.2}
Let $a$ and $b$ be positive integers with $a|b$
and let $S$ be a gcd-closed set with $\max_{x\in S}
\{|G_S(x)|\}=2$ and the condition $\mathcal{G}$ being
satisfied. Then $(S^a)|(S^b), (S^a)|[S^b]$ and $[S^a]|[S^b]$
hold in the ring $M_n({\bf Z})$.
\end{theorem}

\begin{theorem}\label{t1.3} Each of the following is true.

{\rm (i).} For any positive integer $b\neq2$, there exist gcd-closed
sets $S_1$ with $|S_1|=4$ and $\max_{x\in S_1}\{|G_{S_1}(x)|\}=2$
and the condition $\mathcal{G}$ not being satisfied such that
$(S_1)|(S_1^b)$ holds in the ring $M_4({\bf Z})$.

{\rm (ii).} For $b=3,$ or any integer $b\ge4$ and
$b\not\equiv 1,5\pmod 6$, there exist gcd-closed
sets $S_2$ with $|S_2|=4$ and $\max_{x\in S_2}\{|G_{S_2}(x)|\}=2$
and the condition $\mathcal{G}$ not being satisfied such that
$(S_2)|[S_2^b]$ holds in the ring $M_4({\bf Z})$.

{\rm (iii).} There exist integers $b>1$ and gcd-closed
sets $S_3$ with $\max_{x\in S_3}\{|G_{S_3}(x)|\}=2$,
$|S_3|\in\{4,5\}$ and the condition $\mathcal{G}$
not being satisfied such that $[S_3]|[S_3^b]$ holds
in the ring $M_n({\bf Z})$.
\end{theorem}

Obviously, letting $a=b$, Theorem \ref{t1.2}
reduces to the main result of Feng, Hong and Zhao
\cite{FHZ-DM09}. The proofs of Theorems \ref{t1.2}
and \ref{t1.3} use combinatorial and number-theoretic
methods. We organize this paper as follows. In
Section 2, we supply some preliminary lemmas that
are needed in the proof of Theorem \ref{t1.2}. In
Section 3, we present the proofs of Theorems \ref{t1.2}
and \ref{t1.3}. The final section is devoted to some remarks.

As in \cite{H-LAA08}, for any permutation $\sigma$
on the set $\{1, ..., n \}$, $(S^a)|(S^b)$ if and only if $(S_{\sigma} ^a)|(S_{\sigma}^b)$.
Likewise, $(S^a)|[S^b]$ if and only if $(S_{\sigma}^a)|[S_{\sigma}^b]$,
and $[S^a]|[S^b]$ if and only if $[S_{\sigma}^a]|[S_{\sigma}^b]$,
where $S_\sigma:=\{x_{\sigma (1)},..., x_{\sigma (n)}\}$.
So, without loss of any generality (WLOG), throughout we always
assume that $S=\{x_1,..., x_n\}$ satisfies $x_1<...<x_n$.

\section{Auxiliary results}
In this section, we supply several lemmas that will be
needed in the proof of Theorem \ref{t1.2}. We begin with
a result due to Bourque and Ligh which gives the inverse
of the power GCD matrix on a gcd-closed set.
\begin{lemma} {\rm \cite[Theorem 3] {BL-LMA93}} \label{L:2.20}
If $S$ is gcd closed and $(f(S))$ is
nonsingular, then for any integers $i$ and
$j$ with $1\le i, j\le n$, we have
$$((f(S))^{-1})_{ij}:={\underset{x_i|x_k\atop x_j|x_k}
{\sum}}\frac{c_{ik}c_{jk}}{\delta_k}$$
with
\begin{align*}
\delta_k:={\underset{d|x_k\atop d\nmid x_t, x_t<x_k}{\sum}}(f*\mu)(d)
\end{align*}
and
\begin{align}\label{eq2.1}
c_{ij}:=\sum _{dx_i|x_j\atop dx_i\nmid x_t, x_t<x_j}\mu(d).
\end{align}
\end{lemma}

\begin{lemma}\label{L:2.1}
If $S$ is gcd closed, then the power GCD matrix $(S^a)$
is nonsingular and for arbitrary integers $i$ and
$j$ with $1\le i, j\le n$, one has
\begin{align*}
((S^a)^{-1})_{ij}:=\sum_{x_i|x_k\atop x_j|x_k}
\frac{c_{i k}c_{j k}}{\alpha _{a,k}}
\end{align*}
with $c_{ij}$ being defined as in {\rm(\ref{eq2.1})}
and $\alpha_{a, k}$ being defined as in {\rm(\ref{eq1.2})}.
\end{lemma}
\begin{proof} From \cite[Example 1 (ii)]{BL-LMA93},
one knows that the power GCD matrix $(S^a)$
is positive definite, and so is nonsingular. Then
Lemma \ref{L:2.20} applied to $f=\xi_a$ gives us
the expected result. 
\end{proof}

We also need the following Hong's formula for
the determinant of the power LCM matrix on a gcd-closed set.
For any positive integer $x$,
the function $\frac{1}{\xi_a}$ is defined by
$\frac{1}{\xi_a}(x)=\frac{1}{x^a}$.

\begin{lemma}{\rm\cite[Lemma 2.1]{H-AA04}}\label{L:2.2}
If $S$ is gcd closed, then
\begin{align}\label{eq2.2}
\det[S^a]=\prod\limits_{k=1}^n x_k^{2a}\beta_{a,k},
\end{align}
where
\begin{align}\label{eq2.3}
\beta_{a,k}:={\underset{{d|x_k\atop d\nmid x_t,x_t<x_k}}{\sum}}\Big(\frac{1}{\xi_a}*\mu\Big)(d).
\end{align}
\end{lemma}

\begin{lemma}\label{L:2.3}
Let $S$ be gcd closed and $\max_{x\in S}\{|G_S(x)|\}=2$.
Let $\alpha _{a,k}$ and $\beta_{a,k}$ be given as
in {\rm(\ref{eq1.2})} and {\rm(\ref{eq2.3})}, respectively.
Then $\alpha_{a,1}=x_1^a$ and $\beta_{a,1}=x_1^{-a}$, and for any integer $m$ with $2\le m\le n$,
we have
\begin{align*}
\alpha _{a,m}=\left\{
\begin{aligned}
&x_m^a-x_{m_0}^a\ &\hbox{if}\ &G_S(x_m)=\{x_{m_0}\},\\
&x_m^a-x_{m_1}^a-x_{m_2}^a+x_{m_3}^a\ &\hbox{if}\ &G_S(x_m)=\{x_{m_1},\ x_{m_2}\}\ \hbox{and}\
x_{m_3}=(x_{m_1},\ x_{m_2})
\end{aligned}
\right.
\end{align*}
and
\begin{align*}
\beta _{a,m}=\left\{
\begin{aligned}
&{x_m^{-a}}-{x_{m_0}^{-a}}\ &\hbox{if}\ &G_S(x_m)=\{x_{m_0}\},\\
&x_m^{-a}-x_{m_1}^{-a}-x_{m_2}^{-a}+x_{m_3}^{-a}\ &\hbox{if}\ &G_S(x_m)=\{x_{m_1},\ x_{m_2}\}\ \hbox{and}\
x_{m_3}=(x_{m_1},\ x_{m_2}).
\end{aligned}
\right.
\end{align*}
\end{lemma}
\begin{proof}
Employing \cite[Theorem 1.2]{H-JA04}, we directly get Lemma \ref{L:2.3}.
\end{proof}

In what follows, we recall several basic results
on the gcd-closed sets.
\begin{lemma}{\rm\cite[Lemma 2.3]{FHZ-DM09}}\label{L:2.4}
Let S be a gcd-closed
set of $n\ge 2$ distinct positive integers and let
$c_{ij}$ be defined as in {\rm(\ref{eq2.1})}. Then
$$
c_{r1}=\left\{\begin{aligned}
{1}& \quad if\ r=1,\\
{0}&\quad otherwise.
\end{aligned}
\right.
$$
If $2\le m\le n$ and $G_S(x_m)=\{x_{m_0}\}$, then
$$
c_{rm}=\left\{
\begin{aligned}
{-1}&\quad if\ r=m_0,\\
{1}&\quad if\ r=m,\\
{0}&\quad otherwise.
\end{aligned}
\right.
$$
If $3\le m\le n$ and $G_S(x_m)=\{x_{m_1},\ x_{m_2}\}$ and $x_{m_3}=(x_{m_1},\ x_{m_2})$, then
$$
c_{rm}=\left\{
\begin{aligned}
{-1}&\quad if \ r=m_1\ or\ r=m_2,\\
{1}&\quad if\ r=m\ or\ m_3,\\
{0}&\quad otherwise.
\end{aligned}
\right.
$$
\end{lemma}

\begin{lemma}{\rm\cite{FHZ-DM09}}\label{L:2.6}
Let $S$ be a gcd-closed set satisfying
$\max_{x\in S}\{|G_S(x)|\}=2$ and let $x\in S$ satisfy
$|G_S(x)|=2$ and $y\in G_S(x)$. Let $z\in S$ be such that
$z|x,\ z\neq x$ and $z\nmid y$. If $A:=\{u\in S: z|u|x,\ u\neq z\}$
satisfies the condition $\mathcal{G}$,
then $[y,z]=x$.



\end{lemma}

\begin{lemma}{\rm\cite[Lemma 2.2]{H-AC06}}\label{L:2.7}
Let $S$ be gcd closed such that $\max_{x\in S}\{|G_S(x)|\}=2$
and $|S|=n$. Let $\beta _{a,k}$ be defined as in {\rm (\ref{eq2.3})}.
Then $\beta _{a,k}\neq0$ for any integer $k$ with $1\le k\le n$.
\end{lemma}

\begin{lemma}\label{L:2.8}
Let $S$ be a gcd-closed set satisfying $\max_{x\in S}
\{|G_S(x)|\}=2$. Then the $a$th power LCM matrix $[S^a]$
is nonsingular and for all integers $i$ and $j$ with
$1\le i, j\le n$, one has
$$([S^a]^{-1})_{ij}:=\frac{1}{x_i^a x_j^a}{\underset{x_i|x_k\atop x_j|x_k}{\sum}}\frac{c_{ik}c_{jk}}{\beta_{a,k}}$$
with $c_{ij}$
being defined as in {\rm (\ref{eq2.1})} and $\beta_{a,k}$
being defined as in {\rm (\ref{eq2.3})}.
\end{lemma}
\begin{proof}
Since ${[x_i,x_j]}^a=x_i^ax_j^a/{(x_i,x_j)}^a$, we have
\begin{align}\label{eq2.4}
[S^a]={\rm diag}(x_1^a,...,x_n^a)\cdot
\Big(\frac{1}{\xi_a}(x_i,x_j)\Big)\cdot{\rm diag}(x_1^a,...,x_n^a).
\end{align}
Hence
\begin{align}\label{eq2.4'}
\det[S^a]=\det\big(\frac{1}{\xi_a}(x_i,x_j)\big)\cdot
\prod\limits_{k=1}^nx_k^{2a}.
\end{align}
Then from (\ref{eq2.2}) and (\ref{eq2.4'}), we can derive that $$\det\big(\frac{1}{\xi_a}(x_i,x_j)\big)
=\prod\limits_{k=1}^n\beta_{a,k}.$$
Lemma \ref{L:2.7} tells us that $\beta_{a,k}\neq0$
for all positive integers $k\ (\le n)$.
So the matrix $\big(\frac{1}{\xi_a}(x_i,x_j)\big)$
is nonsingular.

Now applying Lemma \ref{L:2.20} to $f=\frac{1}{\xi_a}$,
one gets that
\begin{align}\label{eq2.5}
\Big(\Big(\frac{1}{\xi_a}(x_i,x_j)\Big)^{-1}\Big)_{ij}
={\underset{x_i|x_k\atop x_j|x_k}
{\sum}}\frac{c_{ik}c_{jk}}{\beta_{a,k}}.
\end{align}
Using (\ref{eq2.4}) and (\ref{eq2.5}) gives the required result.
\end{proof}

\begin{lemma} \label{L:2.11}
Let $a$ and $b$ be positive integers such that $a|b$. Let $S$
be a gcd-closed set and $x, y, z\in S$ with $G_S(x)=\{y\}$.
Then each of the following is true:

{\rm (i).} {\rm\cite[Lemma 2.5]{Z-IJNT22}}
The integer $x^a-y^a$ divides each of $(x, z)^b-(y, z)^b$
and $[x,z]^b-[y,z]^b$.

{\rm (ii).} {\rm\cite[Lemma 2.8]{ZL-BAMS22}}
If $r\in S$ and $r|x$, then $y^a[z,x]^b-x^a[z,y]^b$
is divisible by $r^a(y^a-x^a)$.
\end{lemma}

\begin{lemma}\label{L:2.9}
Let $a$ and $b$ be positive integers with $a|b$.
Let $S$ be a gcd-closed set with
$\max_{x\in S}\{|G_S(x)|\}=2$ and $z\in S$.
For $x\in S$ with $|G_S(x)|=2$, let
$G_S(x)=\{y_1, y_2\}$ and $y_3:=(y_1, y_2)$.
Assume that the set $\{u\in S: (x,z)|u|x\}$ satisfies
the condition $\mathcal{G}$. Then
each of the following is true:

{\rm (i).} $x^a+y_3^a-y_1^a-y_2^a$ divides each of
$(z,x)^b+(z,y_3)^b-(z,y_1)^b-(z,y_2)^b$ and
$[z,x]^b+[z,y_3]^b-[z,y_1]^b-[z,y_2]^b$.

{\rm (ii).} For any $r\in S$ with $r|x$,
$x^a[z,y_3]^b+y_3^a[z,x]^b-y_1^a[z,y_2]^b-y_2^a[z,y_1]^b$
is divisible by $r^a(x^a+y_3^a-y_1^a-y_2^a)$.
\end{lemma}

\begin{proof}
Let $d=(x,z)$. If $x|z$, then
$$
(z,x)^b+(z,y_3)^b-(z,y_1)^b-(z,y_2)^b=x^b+y_3^b-y_1^b-y_2^b,
$$
$$
[z,x]^b+[z,y_3]^b-[z,y_1]^b-[z,y_2]^b=z^b+z^b-z^b-z^b=0
$$
and
$$
x^a[z,y_3]^b+y_3^a[z,x]^b-y_1^a[z,y_2]^b-y_2^a[z,y_1]^b=(x^a+y_3^a-y_1^a-y_2^a)z^b.
$$
Since $G_S(x)=\{y_1,\ y_2\}$, $y_3:=(y_1,\ y_2)$
and $x$ satisfies the condition $\mathcal{G}$,
it follows that $xy_3=y_1y_2$. This implies that
for any positive integer $l$, one has
\begin{align*}
x^l+y_3^l-y_1^l-y_2^l
=(y_2^l-y_3^l)\Big(\Big(\frac{x}{y_2}\Big)^l-1\Big).
\end{align*}
So
\begin{align*}
x^b+y_3^b-y_1^b-y_2^b
=(y_2^b-y_3^b)\Big(\Big(\frac{x}{y_2}\Big)^b-1\Big).
\end{align*}
and
\begin{align}\label{*}
x^a+y_3^a-y_1^a-y_2^a
=(y_2^a-y_3^a)\Big(\Big(\frac{x}{y_2}\Big)^a-1\Big).
\end{align}
Since $a|b$, we have
$$
\frac{x^b+y_3^b-y_1^b-y_2^b}{x^a+y_3^a-y_1^a-y_2^a}
=\frac{y_2^b-y_3^b}{y_2^a-y_3^a}
\frac{\Big(\frac{x}{y_2}\Big)^b-1}{\Big(\frac{x}{y_2}\Big)^a-1}\in{\bf Z}.
$$
The statements for parts (i)
and (ii) are clearly true. In what follows, we let $x\nmid z$.
Then $d<x$ and $d\in S$ since $S$ is gcd closed.
The conditions $d|x$ and $G_S(x)=\{y_1, y_2\}$ yield
that either $d|y_1$ or $d|y_2$. One needs only to
consider the following two cases.

{\sc Case 1}. $d|y_1$ and $d|y_2$. Then $d|y_3$. Since
$d|z$, one has $d|(y_3,z)$. However, $y_3|y_1|x$ and
$y_3|y_2|x$. One then derives that $(y_3,z)|(y_1,z)|(x,z)=d$
and $(y_3,z)|(y_2,z)|(x,z)=d$. Then
$(y_3,z)=(y_1,z)=(y_2,z)=(x,z)$ which infers that
$(z,x)^b+(z,y_3)^b-(z,y_1)^b-(z,y_2)^b=0$. Hence
$x^a+y_3^a-y_1^a-y_2^a$ divides
$(z,x)^b+(z,y_3)^b-(z,y_1)^b-(z,y_2)^b$. So the
first statement for part (i) is true in this case.
Moreover, one has
\begin{align}
&[z,x]^b+[z,y_3]^b-[z,y_2]^b-[z,y_1]^b\notag\\
=&\frac{z^bx^b}{(z,x)^b}+\frac{z^by_3^b}{(z,y_3)^b}
-\frac{z^by_2^b}{(z,y_2)^b}-\frac{z^by_1^b}{(z,y_1)^b}\notag \\
=&\frac{z^b}{(z,x)^b}(x^b+y_3^b-y_1^b-y_2^b)\notag\\
=&\frac{z^b}{(z,x)^b}(y_2^b-y_3^b)\Big(\Big(\frac{x}{y_2}\Big)^b-1\Big)
\label{eq2.6}
\end{align}
and
\begin{align}
&x^a[z,y_3]^b+y_3^a[z,x]^b-y_1^a[z,y_2]^b-y_2^a[z,y_1]^b\notag\\
=&x^a\frac{z^by_3^b}{(z,y_3)^b}+y_3^a\frac{z^bx^b}{(z,x)^b}
-y_1^a\frac{z^by_2^b}{(z,y_2)^b}-y_2^a\frac{z^by_1^b}{(z,y_1)^b}\notag \\
=&\frac{z^b}{(z,x)^b}(x^ay_3^b+y_3^ax^b-y_1^ay_2^b-y_2^ay_1^b)\notag\\
=&\frac{z^b}{(z,x)^b}x^ay_3^a(x^{b-a}+y_3^{b-a}-y_1^{b-a}-y_2^{b-a})\notag\\
=&\frac{z^b}{(z,x)^b}x^ay_3^a(y_2^{b-a}-y_3^{b-a})
\Big(\Big(\frac{x}{y_2}\Big)^{b-a}-1\Big). \label{eq2.7}
\end{align}
It follows from (\ref{*}) and (\ref{eq2.6}) that
\begin{align}\label{eq2.8'}
\frac{[z,x]^b+[z,y_3]^b-[z,y_2]^b-[z,y_1]^b}{x^a+y_3^a-y_1^a-y_2^a}
={\Big(\frac{z}{(z,x)}\Big)}^b\cdot\frac{y_2^b-y_3^b}{y_2^a-y_3^a}\cdot
\frac{\big(\frac{x}{y_2}\big)^b-1}{\big(\frac{x}{y_2}\big)^a-1}.
\end{align}
And from (\ref{*}) and (\ref{eq2.7}), one derives that
\begin{align}\label{eq2.8'''}
&\frac{x^a[z,y_3]^b+y_3^a[z,x]^b-y_1^a[z,y_2]^b
-y_2^a[z,y_1]^b}{r^a(x^a+y_3^a-y_1^a-y_2^a)}\notag\\
=&{\Big(\frac{z}{(z,x)}\Big)}^by_3^a\cdot
\Big(\frac{x}{r}\Big)^a\frac{y_2^{b-a}-y_3^{b-a}}{y_2^a-y_3^a}
\cdot\frac{\big(\frac{x}{y_2}\big)^{b-a}-1}{\big(\frac{x}{y_2}\big)^a-1}.
\end{align}

Since $(z,x)|z$, $r|x$, $y_2|x$ and $a|b$, one knows that
all the rational numbers
$${\Big(\frac{z}{(z,x)}\Big)}^b,\ \frac{y_2^b-y_3^b}{y_2^a-y_3^a},\
\frac{\big(\frac{x}{y_2}\big)^b-1}{\big(\frac{x}{y_2}\big)^a-1},\ {\Big(\frac{x}{r}\Big)}^a,\
\frac{y_2^{b-a}-y_3^{b-a}}{y_2^a-y_3^a}\ {\rm and} \
\frac{\big(\frac{x}{y_2}\big)^{b-a}-1}{\big(\frac{x}{y_2}\big)^a-1}$$
are integers. It then follows from (\ref{eq2.8'})
and (\ref{eq2.8'''}) that
$$
\frac{[z,x]^b+[z,y_3]^b-[z,y_2]^b-[z,y_1]^b}
{x^a+y_3^a-y_1^a-y_2^a}\in {\bf Z},
$$
and
$$\frac{x^a[z,y_3]^b+y_3^a[z,x]^b-y_1^a[z,y_2]^b-y_2^a[z,y_1]^b}
{r^a(x^a+y_3^a-y_1^a-y_2^a)}\in {\bf Z}.
$$
In other words, $x^a+y_3^a-y_1^a-y_2^a$ divides
$[z,x]^b+[z,y_3]^b-[z,y_1]^b-[z,y_2]^b$,
and $x^a[z,y_3]^b+y_3^a[z,x]^b-y_1^a[z,y_2]^b-y_2^a[z,y_1]^b$
is divisible by $r^a(x^a+y_3^a-y_1^a-y_2^a)$.
So the second statement of part (i) and part (ii) are
true in this case. Lemma \ref{L:2.9} is proved in this case.

{\sc Case 2}. $d$ divides exactly one of $y_1$ and $y_2$.
WLOG, one may let $d|y_1$ and
$d\nmid y_2$. Since the set $\{u\in S: (x,z)|u|x\}$
satisfies the condition $\mathcal{G}$, $d|x,\ d\neq x,\ y_2\in G_S(x)$
and $d\nmid y_2$, applying Lemma \ref{L:2.6} gives us that
$[d,y_2]=x.$
Likewise, we have
$[d,y_3]=y_1.$
In fact, if $d=y_1$, then by $y_3|y_1$, we know that $y_3|d$
and so $[d,y_3]=d=y_1$. Now we let $d\ne y_1$.
Since $d|y_1, y_3\in G_S(y_1)$, and $d\nmid y_2$ implying
that $d\nmid y_3$, by Lemma \ref{L:2.6} we derive
that $[d,y_3]=y_1$. But $d=(x, z)|z$.
Then one can deduce that
\begin{align}
[z,x]=[z,[d,y_2]]=[[z,d],y_2]=[z,y_2]\ {\rm and}\ [z,y_1]=[z,[d,y_3]]
=[z,y_3].\label{eq2.11}
\end{align}
It readily follows from (\ref{eq2.11}) that
\begin{align}\label{eq2.11'}
[z,x]^b+[z,y_3]^b-[z,y_1]^b-[z,y_2]^b=0.
\end{align}

On the one hand, since $y_1|x$, we have $(z,y_1)|(z,x)$.
On the other hand, $(z,x)=d|y_1$ together with $d=(x,z)|z$
yields that $(z,x)|(z,y_1)$. Therefore
\begin{align}
(z,x)=(z,y_1).\label{eq2.12}
\end{align}
Since $xy_3=y_1y_2$, by (\ref{eq2.11}) and
(\ref{eq2.12}), we have
\begin{align}\label{eq2.13}
&(z,x)^b+(z,y_3)^b-(z,y_1)^b-(z,y_2)^b\notag\\
=&(z,y_3)^b-(z,y_2)^b\notag\\
=&\frac{z^by_3^b}{[z,y_3]^b}-\frac{z^by_2^b}{[z,y_2]^b}\notag\\
=&\frac{z^by_3^b}{[z,y_1]^b}-\frac{z^by_2^b}{[z,x]^b}\notag\\
=&(z,y_1)^b\frac{y_3^b}{y_1^b}-(z,x)^b\frac{y_2^b}{x^b}=0.
\end{align}
So by (\ref{eq2.11'}) and (\ref{eq2.13}), we know that
$x^a+y_3^a-y_1^a-y_2^a$ divides both of
$(z,x)^b+(z,y_3)^b-(z,y_1)^b-(z,y_2)^b$ and
$[z,x]^b+[z,y_3]^b-[z,y_1]^b-[z,y_2]^b$.
Part (i) holds in this case.

Likewise, by (\ref{eq2.11}) and (\ref{eq2.12}),
we can deduce that
\begin{align}
&x^a[z,y_3]^b+y_3^a[z,x]^b-y_1^a[z,y_2]^b-y_2^a[z,y_1]^b\notag\\
=&x^a[z,y_1]^b+y_3^a[z,x]^b-y_1^a[z,x]^b-y_2^a[z,y_1]^b\notag\\
=&x^a\frac{z^by_1^b}{(z,y_1)^b}+y_3^a\frac{z^bx^b}{(z,x)^b}
-y_1^a\frac{z^bx^b}{(z,x)^b}-y_2^a\frac{z^by_1^b}{(z,y_1)^b}\notag\\
=&\frac{z^b}{(z,x)^b}(x^ay_1^b+y_3^ax^b-y_1^ax^b-y_2^ay_1^b)\notag\\
=&\frac{z^b}{(z,x)^b}(x^ay_1^b+y_3^ax^b-y_1^ax^b-x^a y_3^a y_1^{b-a})\notag\\
=&\Big(\frac{z}{(z,x)}\Big)^bx^a(y_1^{b-a}-x^{b-a})(y_1^a-y_3^a)\notag\\
=&\Big(\frac{z}{(z,x)}\Big)^bx^ay_1^{b-a}
\Big(1-\Big(\frac{x}{y_1}\Big)^{b-a}\Big)(y_1^a-y_3^a).\label{eq2.14}
\end{align}
Since $xy_3=y_1y_2$, we have
\begin{align}
x^a+y_3^a-y_1^a-y_2^a=\Big(\Big(\frac{x}{y_1}\Big)^a
-1\Big)(y_1^a-y_3^a)\label{e2.15}.
\end{align}
So by (\ref{eq2.14}) and (\ref{e2.15}), one has
$$
\begin{aligned}
&\frac{x^a[z,y_3]^b+y_3^a[z,x]^b-y_1^a[z,y_2]^b-y_2^a[z,y_1]^b}
{r^a(x^a+y_3^a-y_1^a-y_2^a)}\\
=&-\Big(\frac{x}{r}\Big)^a{\Big(\frac{z}{(z,x)}\Big)}^by_1^{b-a}
\frac{\big(\frac{x}{y_1}\big)^{b-a}-1}{\big(\frac{x}{y_1}\big)^a-1}.
\end{aligned}
$$
But the condition $a|(b-a)$ implies that
$$\frac{\big(\frac{x}{y_1}\big)^{b-a}-1}{\big(\frac{x}{y_1}\big)^a-1}\in {\bf Z}.$$
So
$$
\frac{x^a[z,y_3]^b+y_3^a[z,x]^b-y_1^a[z,y_2]^b
-y_2^a[z,y_1]^b}{r^a(x^a+y_3^a-y_1^a-y_2^a)}\in {\bf Z}.
$$
That is, $r^a(x^a+y_3^a-y_1^a-y_2^a)$ divides
$x^a[z,y_3]^b+y_3^a[z,x]^b-y_1^a[z,y_2]^b-y_2^a[z,y_1]^b$
as desired. Thus part (ii) is proved in this case.

This concludes the proof of Lemma \ref{L:2.9}.
\end{proof}

\begin{lemma}\label{L:2.15}
Let $a$ and $b$ be positive integers with $a|b$.
Let $S$ be a gcd-closed set satisfying
$\max_{x\in S}\{|G_S(x)|\}=2$.
If $S$ satisfies the condition $\mathcal{G}$,
then all the elements of the $n$-th column
and the $n$-th row of the matrices $(S^b)(S^a)^{-1}$,
$[S^b](S^a)^{-1}$ and $[S^b][S^a]^{-1}$ are integers.
\end{lemma}
\begin{proof} We divide the proof into the following two cases:

{\sc Case 1.} $1\le i\le n$ and $j=n$. By Lemmas
\ref{L:2.1} and \ref{L:2.8}, we have
\begin{align}
{\left( (S^b)(S^a)^{-1}\right)}_{in}=&\sum_{m=1}^n(x_i, x_m)^b
\sum_{x_m|x_k\atop x_n|x_k}\frac {c_{mk}c_{nk}}{\alpha _{a,k}}\notag\\
=&\frac{1}{\alpha _{a,n}}\sum_{m=1}^n(x_i,x_m)^bc_{mn},\label{eq2.17}
\end{align}
\begin{align}
{\left( [S^b](S^a)^{-1}\right)}_{in}
=&\sum_{m=1}^n[x_i, x_m]^b\sum_{x_m|x_k\atop x_n|x_k}
\frac {c_{mk}c_{nk}}{\alpha _{a,k}}\notag\\
=&\frac{1}{\alpha _{a,n}}\sum_{m=1}^n [x_i,x_m]^bc_{mn}\label{eq2.18}
\end{align}
and
\begin{align}
{\left([S^b][S^a]^{-1}\right)}_{in}
=&\sum_{m=1}^n[x_i, x_m]^b\frac{1}{x_m^ax_n^a}\sum_{x_m|x_k\atop x_n|x_k}
\frac{c_{mk}c_{nk}}{\beta _{a,k}}\notag\\
=&\frac{1}{x_n^a\beta_{a,n}}\sum_{m=1}^n\frac{[x_i,x_m]^bc_{mn}}{x_m^a}.
\label{eq2.19}
\end{align}

If $|G_S(x_n)|=1$,  we may let $G_S(x_n)=\{x_{n_1}\}$.
Then by (\ref{eq2.17}) to (\ref{eq2.19}), Lemmas
\ref{L:2.3}, \ref{L:2.4} and \ref{L:2.11}, one deduces that
$${\left( (S^b)(S^a)^{-1}\right)}_{in}
={\frac{(x_i,x_n)^b-(x_i,x_{n_1})^b}{x_n^a-x_{n_1}^a}}\in {\bf Z},$$
$$
{\left([S^b])(S^a)^{-1}\right)}_{in}
={\frac{[x_i,x_n]^b-[x_i,x_{n_1}]^b}{x_n^a-x_{n_1}^a}}\in {\bf Z}
$$
and
$$
{\left([S^b])[S^a]^{-1}\right)}_{in}
=\dfrac{x_{n_1}^a[x_i,x_n]^b-x_n^a[x_i,x_{n_1}]^b}{x_n^a(x_{n_1}^a-x_n^a)}\in {\bf Z}
$$
as required.

If $|G_S(x_n)|=2$, we may let $G_S(x_n)=\{x_{n_1},x_{n_2} \}$ and $x_{n_3}=(x_{n_1},x_{n_2})$. Then $x_{n} x_{n_3}=x_{n_1}x_{n_2}$.
Since $S$ satisfies the condition $\mathcal{G}$, by
(\ref{eq2.17}) to (\ref{eq2.19}), Lemmas \ref{L:2.3},
\ref{L:2.4} and \ref{L:2.9}, one derives that
\begin{align*}
{\left((S^b))(S^a)^{-1}\right)}_{in}
=\frac{(x_i,x_n)^b-(x_i,x_{n_1})^b-(x_i,x_{n_2})^b
+(x_i,x_{n_3})^b}{x_n^a-x_{n_1}^a-x_{n_2}^a+x_{n_3}^a}\in {\bf Z},
\end{align*}

\begin{align*}
{\left([S^b])(S^a)^{-1}\right)}_{in}
=\frac{[x_i,x_n]^b-[x_i, x_{n_1}]^b-[x_i, x_{n_2}]^b+[x_i,
x_{n_3}]^b}{x_n^a-x_{n_1}^a-x_{n_2}^a+x_{n_3}^a}\in {\bf Z}
\end{align*}
and
\begin{align*}
{\left([S^b])[S^a]^{-1}\right)}_{in}
&=\frac{\frac{[x_n,x_i]^b}{x_n^a}-\frac{[x_{n_1},x_i]^b}
{x_{n_1}^a}-\frac{[x_{n_2},x_i]^b}{x_{n_2}^a}+
\frac{[x_{n_3},x_i]^b}{x_{n_3}^a}}{x_n^a\big(\frac{1}
{x_n^a}-\frac{1}{x_{n_1}^a}-\frac{1}{x_{n_2}^a}+\frac{1}{x_{n_3}^a}\big)}\\
&=\frac{x_n^a[x_i,x_{n_3}]^b+x_{n_3}^a[x_i,x_n]^b-x_{n_2}^a
[x_i,x_{n_1}]^b-x_{n_1}^a[x_i,x_{n_2}]^b}
{x_n^a({x_n^a+x_{n_3}^a-x_{n_1}^a-x_{n_2}^a})}\in {\bf Z}
\end{align*}
as required. So Lemma \ref{L:2.15} is proved in Case 1.

{\sc Case 2.} $i=n$ and $1\le j\le n-1$. By Lemmas
\ref{L:2.1} and \ref{L:2.8}, one has
\begin{align*}
{\left( (S^b)(S^a)^{-1}\right)}_{nj}
=&\sum_{m=1}^n(x_n, x_m)^b\sum_{x_m|x_k\atop x_j|x_k}
\frac {c_{mk}c_{jk}}{\alpha _{a,k}}\\
=&\sum_{x_j|x_k}\frac{c_{jk}}{\alpha _{a,k}}
\sum_{x_m|x_k}c_{mk}(x_m,x_n)^b\\
:=&\sum_{x_j|x_k}c_{jk}\omega_k,
\end{align*}

\begin{align*}
{\left( [S^b](S^a)^{-1}\right)}_{nj}
=&\sum_{m=1}^n[x_n, x_m]^b\sum_{x_m|x_k\atop x_j|x_k}\frac {c_{mk}c_{jk}}{\alpha _{a,k}}\\
=&\sum_{x_j|x_k}\frac{c_{jk}}{\alpha _{a,k}}\sum_{x_m|x_k}c_{mk}[x_m,x_n]^b\\
:=&\sum_{x_j|x_k}c_{jk}\gamma_k
\end{align*}
and
\begin{align*}
{\left( [S^b][S^a]^{-1}\right)}_{nj}
=&\sum_{m=1}^n[x_n, x_m]^b\frac{1}{x_m^ax_j^a}\sum_{x_m|x_k\atop x_j|x_k}\frac {c_{mk}c_{jk}}{\beta _{a,k}}\\
=&\sum_{x_j|x_k}\frac{c_{jk}}{x_j^a\beta _{a,k}}\sum_{x_m|x_k}\frac{1}{x_m^a}c_{mk}[x_m,x_n]^b\\
:=&\sum_{x_j|x_k}c_{jk}\eta_k.
\end{align*}
Claim that for any positive integer $k$ with $x_j|x_k$, one has
$\omega_k\in {\bf Z}$, $\gamma_k\in {\bf Z}$ and $\eta_k \in {\bf Z}$.

If $k=1$, then we must have $m=j=1$. In this case, one has
$$\omega_1=\frac{1}{\alpha _{a,1}}\cdot c_{11}\cdot(x_1,x_n)^b
=\frac{(x_1,x_n)^b}{x_1^a}=x_1^{b-a}\in {\bf Z},$$
$$\gamma_1=\frac{1}{\alpha _{a,1}}\cdot c_{11}\cdot[x_1,x_n]^b
=\frac{[x_1,x_n]^b}{x_1^a}=\frac{x_1^{b-a}x_n^b}{{(x_1,x_n)}^b}\in {\bf Z}$$
and
$$\eta_1=\frac{1}{\beta _{a,1}}\cdot\frac{1}
{x_1^{2a}}\cdot c_{11}\cdot[x_1,x_n]^b
=\frac{[x_1,x_n]^b}{x_1^a}\in {\bf Z}$$
since $\alpha_{a,1}=x_1^a$ and $\beta_{a,1}=x_1^{-a}$.
So the claim is true when $k=1$.

Now let $k>1$. If $|G_S(x_k)|=1$, one can set
$G_S(x_k)=\{x_{k_1}\}$ with $1\le k_1<k$.
By Lemmas \ref{L:2.3}, \ref{L:2.4} and \ref{L:2.11}, we have
$$
\omega_k=\frac{1}{\alpha _{a,k}}\sum_{x_m|x_k}c_{mk}(x_m,x_n)^b
={\frac{(x_k,x_n)^b-(x_{k_1},x_n)^b}{x_k^a-x_{k_1}^a}}\in {\bf Z},
$$

$$
\gamma_k=\frac{1}{\alpha _{a,k}}\sum_{x_m|x_k}c_{mk}[x_m,x_n]^b
={\frac{[x_k,x_n]^b-[x_{k_1},x_n]^b}{x_k^a-x_{k_1}^a}}\in {\bf Z}
$$
and
$$
\eta_k={\frac{1}{x_j^a\beta _{a,k}}\sum_{x_m|x_k}
\dfrac{1}{x_m^a}c_{mk}[x_m,x_n]^b}
={\frac{x^a_{k_1}[x_k,x_n]^b-x_k^a[x_{k_1},x_n]^b}
{x_j^a(x_{k_1}^a-x_k^a)}}\in {\bf Z}
$$
as claimed.
So we need only to treat the remaining case
$|G_S(x_k)|=2$. Now let $G_S(x_k)=\{x_{k_1},x_{k_2} \}$
and $x_{k_3}:=(x_{k_1},x_{k_2})$. Then by Lemmas
\ref{L:2.3}, \ref{L:2.4} and \ref{L:2.9}, we have
$$
\begin{aligned}
\omega_k
={\frac{(x_k,x_n)^b-(x_{k_1},x_n)^b-(x_{k_2},x_n)^b
+(x_{k_3},x_n)^b}{x_k^a-x_{k_1}^a-x_{k_2}^a+x_{k_3}^a}}\in {\bf Z},
\end{aligned}
$$

$$
\begin{aligned}
\gamma_k={\frac{[x_k,x_n]^b-[x_{k_1},x_n]^b-[x_{k_2},x_n]^b
+[x_{k_3},x_n]^b}{x_k^a-x_{k_1}^a-x_{k_2}^a+x_{k_3}^a}}\in {\bf Z}
\end{aligned}
$$
and
\begin{align*}
\eta_k
=&\frac{\frac{[x_k,x_n]^b}{x_k^a}-\frac{[x_{k_1},x_n]^b}
{x_{k_1}^a}-\frac{[x_{k_2},x_n]^b}{x_{k_2}^a}+
\frac{[x_{k_3},x_n]^b}{x_{k_3}^a}}{x_j^a\big(\frac{1}{x_k^a}
-\frac{1}{x_{k_1}^a}-\frac{1}{x_{k_2}^a}+\frac{1}{x_{k_3}^a}\big)}\\
=&\frac{x_k^a[x_n,x_{k_3}]^b+x_{k_3}^a[x_n,x_k]^b-x_{k_2}^a
[x_n,x_{k_1}]^b-x_{k_1}^a[x_n,x_{k_2}]^b}
{x_j^a({x_k^a+x_{k_3}^a-x_{k_1}^a-x_{k_2}^a})}\in {\bf Z}
\end{align*}
as desired.

This completes the proof of Case 2 and that of Lemma \ref{L:2.15}.
\end{proof}

\begin{lemma} {\rm\cite[Theorem 1.3]{Z-IJNT22}}
{\rm\cite[Theorem 1.1]{ZL-BAMS22}} \label{L:2.13}
Let $a$ and $b$ be positive integers with $a|b$
and let $S$ be a gcd-closed set satisfying $\max_{x\in S}
\{|G_S(x)|\}=1$. Then in the ring $M_{|S|}({\bf Z})$,
we have $(S^a)|(S^b), (S^a)|[S^b]$ and $[S^a]|[S^b]$.
\end{lemma}

Finally, we can use Lemma \ref{L:2.15} to show the
following main result of this section.

\begin{lemma}\label{L:2.16}
Let $S$ be a gcd-closed set satisfying $\max_{x\in S}
\{|G_S(x)|\}=2$ and let $a$ and $b$ be positive
integers such that $a|b$. Let $S_0:=S\setminus\{\max(S)\}$.
If $S$ satisfies the condition $\mathcal{G}$, then
$$(S^b)(S^a)^{-1}\in M_n({\bf Z})
\Leftrightarrow(S_0^b)(S_0^a)^{-1}\in M_{n-1}({\bf Z}),$$

$$[S^b](S^a)^{-1}\in M_n({\bf Z})\Leftrightarrow[S_0^b]
(S_0^a)^{-1}\in M_{n-1}({\bf Z})$$
and
$$
[S^b][S^a]^{-1}\in M_n({\bf Z})\Leftrightarrow
[S_0^b][S_0^a]^{-1}\in M_{n-1}({\bf Z}).
$$
\end{lemma}
\begin{proof} It is clear that
$S_0:=S\setminus \{x_n\}=\{x_1, ..., x_{n-1}\}$.
At first, by Lemma \ref{L:2.15}, one knows
that all the elements of the $n$-th column and the
$n$-th row of the matrices $(S^b)(S^a)^{-1}$,
$[S^b](S^a)^{-1}$ and $[S^b][S^a]^{-1}$ are integers.
So it is sufficient to show that
$\forall\ i,j$ ($1\le i,j\le n-1$), one has
\begin{align}\label{eq2.20}
\mathcal{D}_{ij}:={\left((S^b)(S^a)^{-1}\right)}_{ij}
-{\left((S_0^b)(S_0^a)^{-1}\right)}_{ij}\in {\bf Z},
\end{align}
\begin{align}\label{eq2.21}
\mathcal{E}_{ij}:={\left([S^b](S^a)^{-1}\right)}_{ij}
-{\left([S_0^b](S_0^a)^{-1}\right)}_{ij}\in {\bf Z}
\end{align}
and
\begin{align}\label{eq2.22}
\mathcal{F}_{ij}:={\left([S^b][S^a]^{-1}\right)}_{ij}
-{\left([S_0^b][S_0^a]^{-1}\right)}_{ij}\in {\bf Z}.
\end{align}

For this, we define the following function:
$$
e_{\mathcal u\mathcal v}:=\left\{
\begin{aligned}
1&\quad {\rm if} \ x_{\mathcal v}|x_{\mathcal u}\\
0&\quad {\rm if} \ x_{\mathcal v}\nmid x_{\mathcal u}
\end{aligned}
\right.
$$
for any positive integers $\mathcal u,\mathcal v$
($\mathcal u,\mathcal v\le n$).
Then for any positive integer $m$ ($\le n-1$), we have $e_{nm}=1$ if $x_m|x_n$, and
$e_{nm}=0$ otherwise. We can deduce that
\begin{align}
\mathcal{D}_{ij}
=&\sum_{m=1}^n(x_i,x_m)^b\sum_{x_m|x_k\atop x_j|x_k}
\frac{c_{mk}c_{jk}}{\alpha_{a,k}}
-\sum_{m=1}^{n-1}(x_i,x_m)^b\sum _{x_m|x_k\atop x_j|x_k,
x_k\neq x_n}\frac{c_{mk}c_{jk}}{\alpha _{a,k}}\notag\\
=&\frac{c_{nn}c_{jn}}{\alpha _{a,n}}(x_i,x_n)^b e_{nj}+
\sum_{m=1}^{n-1}\frac{c_{mn}c_{jn}}{\alpha_{a,n}}
(x_i,x_m)^b e_{nj}e_{nm} \notag\\
=&e_{nj} \frac{c_{jn}}{\alpha _{a,n}}\Big((x_i,x_n)^b
+\sum_{m=1}^{n-1}(x_i,x_m)^b c_{mn}e_{nm}\Big)\notag\\
:=& e_{nj} D_{ij}. \label{eq2.23}
\end{align}

Likewise, we have
\begin{align}
\mathcal{E}_{ij}
=&\sum_{m=1}^n[x_i,x_m]^b\sum_{x_m|x_k\atop x_j|x_k}
\frac{c_{mk}c_{jk}}{\alpha _{a,k}}
-\sum_{m=1}^{n-1}[x_i,x_m]^b\sum _{x_m|x_k\atop x_j|x_k,\
x_k\neq x_n}\frac{c_{mk}c_{jk}}{\alpha _{a,k}}\notag\\
=&\frac{c_{nn}c_{jn}}{\alpha _{a,n}}[x_i,x_n]^b e_{nj}+
\sum_{m=1}^{n-1}\frac{c_{mn}c_{jn}}
{\alpha _{a,n}}[x_i,x_m]^b e_{nj}e_{nm} \notag\\
=&e_{nj}\frac{c_{jn}}{\alpha _{a,n}}\Big([x_i,x_n]^b +
\sum_{m=1}^{n-1}[x_i,x_m]^b c_{mn}e_{nm}\Big)\notag\\
:=& e_{nj} E_{ij}. \label{eq2.24}
\end{align}
and
\begin{align}
\mathcal{F}_{ij}
=&\sum_{m=1}^n[x_i,x_m]^b\sum_{x_m|x_k\atop x_j|x_k}
\frac{c_{mk}c_{jk}}{x_m^ax_j^a\beta _{a,k}}
-\sum_{m=1}^{n-1}[x_i,x_m]^b\sum _{x_m|x_k\atop x_j|x_k,\
x_k\neq x_n}\frac{c_{mk}c_{jk}}{x_m^ax_j^a\beta _{a,k}}\notag\\
=&\frac{c_{nn}c_{jn}}{x_n^ax_j^a\beta _{a,n}}[x_i,x_n]^b e_{nj}+
\sum_{m=1}^{n-1}\frac{c_{mn}c_{jn}}{x_m^ax_j^a\beta _{a,n}}
[x_i,x_m]^b e_{nj}e_{nm} \notag\\
=&e_{nj}\frac{c_{jn}}{x_j^a\beta _{a,n}}
\Big(\frac{[x_i,x_n]^b}{x_n^a}+\sum_{m=1}^{n-1}
\frac{[x_i,x_m]^bc_{mn}e_{nm}}{x_m^a} \Big)\notag\\
:=&e_{nj} F_{ij}. \label{eq2.25}
\end{align}

In what follows, we show that $D_{ij}\in {\bf Z}$,
$E_{ij}\in {\bf Z}$ and $F_{ij}\in {\bf Z}$.
Consider the following two cases:

{\sc Case 1.} $|G_S(x_n)|=1$. One may let $G_S(x_n)=\{x_{n_0}\}$,
By Lemma \ref{L:2.3}, one has
$$\alpha_{a,n}=x_n^a-x_{n_0}^a \ {\rm and} \
\beta_{a,n}=x_n^{-a}-x_{n_0}^{-a}.$$
However, for any positive integer $m$ ($\le n-1$),
by Lemma \ref{L:2.4}, $c_{mn}=-1$ if $m=n_0$ and $c_{mn}=0$ otherwise.
So from (\ref{eq2.23}) to (\ref{eq2.25}) and Lemma \ref{L:2.11}
one can derive that
\begin{align}\label{eq2.26}
{D_{ij}=\frac{(x_i,x_n)^b-(x_i,x_{n_0})^b}
{x_n^a-x_{n_0}^a}\cdot c_{jn}}\in {\bf Z},
\end{align}

\begin{align}\label{eq2.27}
{E_{ij}=\frac{[x_i,x_n]^b-[x_i,x_{n_0}]^b}
{x_n^a-x_{n_0}^a}\cdot c_{jn}}\in {\bf Z}
\end{align}
and
\begin{align}\label{eq2.28}
{F_{ij}=\frac{x_{n_0}^a[x_i,x_n]^b-x_n^a[x_i,x_{n_0}]^b}
{x_j^a(x_{n_0}^a-x_n^a)}\cdot c_{jn}}\in {\bf Z}.
\end{align}
Since $e_{nj}\in\{0,1\}$, (\ref{eq2.20}) to (\ref{eq2.22}) follow immediately from
(\ref{eq2.23}) and (\ref{eq2.26}),
(\ref{eq2.24}) and (\ref{eq2.27}), and
(\ref{eq2.25}) and (\ref{eq2.28}), respectively.

{\sc Case 2.} $|G_S(x_n)|=2$. Let $G_S(x_n)=\{x_{n_1},x_{n_2} \}$
and $x_{n_3}=(x_{n_1},x_{n_2})$. It then follows from
(\ref{eq2.23}) to (\ref{eq2.25}) and Lemmas \ref{L:2.3},
\ref{L:2.4} and \ref{L:2.9} that
$$
{D_{ij}=\frac{(x_i,x_n)^b-(x_i,x_{n_1})^b-(x_i,x_{n_2})^b+(x_i,x_{n_3})^b}
{x_n^a-x_{n_1}^a-x_{n_2}^a+x_{n_3}^a}\cdot c_{jn}}\in {\bf Z},
$$

$$
{E_{ij}=\frac{[x_i,x_n]^b-[x_i,x_{n_1}]^b-[x_i,x_{n_2}]^b+[x_i,x_{n_3}]^b}
{x_n^a-x_{n_1}^a-x_{n_2}^a+x_{n_3}^a}\cdot c_{jn}}\in {\bf Z}
$$
and
\begin{align*}
F_{ij}&=\frac{c_{jn}}{x_j^a\alpha_{a,n}}
\Big(\frac{[x_n,x_i]^b}{x_n^a}-\frac{[x_{n_1},x_i]^b}{x_{n_1}^a}
-\frac{[x_{n_2},x_i]^b}{x_{n_2}^a}+
\frac{[x_{n_3},x_i]^b}{x_{n_3}^a}\Big)\\
&=c_{jn}\cdot\frac{x_n^a[x_i,x_{n_3}]^b+x_{n_3}^a[x_i,x_n]^b
-x_{n_2}^a[x_i,x_{n_1}]^b-x_{n_1}^a[x_i,x_{n_2}]^b}
{x_j^a({x_n^a+x_{n_3}^a-x_{n_1}^a-x_{n_2}^a})}\in {\bf Z}.
\end{align*}
Hence (\ref{eq2.20}) to (\ref{eq2.22})
hold in this case.

This finishes the proof of Lemma \ref{L:2.16}.
\end{proof}

\section{Proofs of Theorems \ref{t1.2} and \ref{t1.3}}

In this section, we first use the lemmas presented in the
previous section to show Theorem \ref{t1.2}.\\

\noindent{\it Proof of Theorem \ref{t1.2}.}
We prove Theorem \ref{t1.2} by using induction on $n=|S|$.

Let $n\le 3$. Since $S$ is gcd closed, the set $S$
satisfies $\max_{x\in S}\{|G_S(x)|\}=1$.
It then follows immediately from Lemma \ref{L:2.13}
that Theorem \ref{t1.2} (i) holds when $n\le 3$.

Now let $n\ge 4$. Assume that the result is true for the
$n-1$ case. In what follows, we show that the result is
true for the $n$ case. Since $S$ is a gcd-closed set with
$\max_{x\in S}\{|G_S(x)|\}=2$ and $S$ satisfies the
condition $\mathcal{G}$, it follows that
$S_0:=S\setminus\{\max(S)\}$ is gcd closed and
$\max_{x\in S_0}\{|G_{S_0}(x)|\}\le 2$
and $S_0$ also satisfies the condition $\mathcal{G}$.
One asserts that
\begin{align}\label{eq3.1}
(S_0^b)(S_0^a)^{-1}\in M_{n-1}({\bf Z}),
[S_0^b](S_0^a)^{-1}\in M_{n-1}({\bf Z}) \ {\rm and}
\ [S_0^b][S_0^a]^{-1}\in M_{n-1}({\bf Z}).
\end{align}
We divide its proof into the following two cases.

{\sc Case 1.} $\max_{x\in S_0}\{|G_{S_0}(x)|\}=1$.
Then by Lemma \ref{L:2.13}, we know that (3.1) holds
in this case.

{\sc Case 2.} $\max_{x\in S_0}\{|G_{S_0}(x)|\}=2$.
Then it follows from the inductive hypothesis that
(3.1) is true. The assertion is proved in this case.

Now we can apply Lemma \ref{L:2.16}. One arrives at
$$
(S^b)(S^a)^{-1}\in M_n({\bf Z}),
[S^b](S^a)^{-1}\in M_n({\bf Z}) \ {\rm and}
\ [S^b][S^a]^{-1}\in M_n({\bf Z}).
$$
In other words, in the ring $M_n({\bf Z})$,
we have $(S^a)|(S^b), (S^a)|[S^b]$ and
$[S^a]|[S^b]$ as desired. Hence Theorem
\ref{t1.2} is true for the $n$ case.
So Theorem \ref{t1.2} is proved. \hfill$\Box$\\

Finally, we give the proof of Theorem \ref{t1.3}.\\

\noindent{\it Proof of Theorem \ref{t1.3}.}
(i). Let
\begin{align}\label{eq 3.2}
S_1:=\{1,u,v,uvw\}\ {\rm with}\ (u,v)=1\ {\rm and}\ w>1.
\end{align}
Evidently, $\max_{x\in S_1}\{|G_{S_1}(x)|\}=2$
and the condition $\mathcal{G}$ is not satisfied.
We can compute and get that
\begin{align}\label{eq 3.3}
(S_1)^{-1}(S_1^b)=&
\begin{pmatrix}
1&1&1&1\\1&u&1&u\\1&1&v&v\\1&u&v&uvw
\end{pmatrix}^{-1}\cdot
\begin{pmatrix}
1&1&1&1\\1&u^b&1&u^b\\1&1&v^b&v^b\\1&u^b&v^b&(uvw)^b
\end{pmatrix}\notag\\
=&\begin{pmatrix}
1&1-\frac{u^b-1}{u-1}&1-\frac{v^b-1}{v-1}&1-\frac{u^b-1}
{u-1}-\frac{v^b-1}{v-1}+\frac{\Delta_b}{\Delta_1}\\
\\0&\frac{u^b-1}{u-1}&0&\frac{u^b-1}{u-1}-\frac{\Delta_b}{\Delta_1}\\
\\0&0&\frac{v^b-1}{v-1}&\frac{v^b-1}{v-1}-\frac{\Delta_b}{\Delta_1}\\
\\0&0&0&\frac{\Delta_b}{\Delta_1}
\end{pmatrix},
\end{align}
where
\begin{align}\label{eq 3.5}
\Delta_b:=(uvw)^b-u^b-v^b+1.
\end{align}
It follows from (\ref{eq 3.3}) that $(S_1)^{-1}(S_1^b)\in M_4({\bf Z})$
if and only if $\frac{\Delta_b}{\Delta_1}\in {\bf Z}$.
Let us continue the proof of part (i) of Theorem \ref{t1.3},
which is divided into four cases.

{\sc Case 1-1.} Picking $u=2,v=3,w=2$, one has
$$
\frac{\Delta_b}{\Delta_1}=\frac{2^b3^b2^b-2^b-3^b+1}
{2\times 3\times 2-2-3+1}
=\frac{2^b(6^b-1)}{2^3}-\frac{3^b-1}{2^3}.
$$
Since $3^2\equiv 1\pmod 8$, we know that if $b$ is
even and $b\ge 4$, then
$\frac{\Delta_b}{\Delta_1}\in {\bf Z}$. Hence
$(S_1)|(S_1^b)$ in this case.

{\sc Case 1-2.}
Taking $u=3,v=4,w=4$, one attains that
$\Delta_1=3\times 4\times 4-3-4+1=42$.

Now let $b\equiv 1\pmod 6$. By Fermat's little theorem,
one knows that $3^6\equiv 1\pmod 7$. One then derives that
$3^b\equiv 3\pmod 7$. Evidently, we have
$3^b\equiv 3\pmod 6$. Thus $3^b\equiv 3\pmod {42}$.
Likewise, we have $4^b\equiv 4\pmod {42}$.
Therefore, $\Delta_b=3^b4^b4^b-3^b-4^b+1\equiv
3\times 4\times 4-3-4+1\equiv 0\pmod {42}$.
Thus $\frac{\Delta_b}{\Delta_1}\in {\bf Z}$
implies that $(S_1)|(S_1^b)$ as desired.

{\sc Case 1-3.} Letting $u=3,v=4,w=2$ gives that
$\Delta_1=3\times 4\times 2-3-4+1=18$.
Let $b\equiv 3\pmod 6$. By Euler's theorem, one has
$4^6\equiv 1\pmod 9$. One can deduce that
$4^b\equiv 10\pmod 9$. Clearly, $4^b\equiv 10\pmod 2$.
It follows that $4^b\equiv 10\pmod {18}.$ Similarly,
we can get that $3^b\equiv 9\pmod{18}$
and $2^b\equiv 8\pmod{18}$.
So $\Delta_b=3^b4^b2^b-3^b-4^b+1
\equiv 9\times 10\times 8-9-10+1\equiv 0\pmod {18}$.
Thus $\frac{\Delta_b}{\Delta_1}\in {\bf Z}$. Hence
$(S_1)|(S_1^b)$ holds in this case.

{\sc Case 1-4.}
Picking $u=2,v=5,w=2$, we have $\Delta_1=2\times 5\times 2-2-5+1=14$.
Let $b\equiv 5\pmod 6$. It follows from Fermat little
theorem that $2^6\equiv 1\pmod 7$. Then one derives that $2^b\equiv4\pmod 7$.
Evidently, $2^b\equiv 4\pmod 2$. Therefore we obtain that $2^b\equiv4\pmod{14}$.
By Euler's theorem, we can directly deduce that $5^6\equiv 1\pmod {14}$.
So $5^b\equiv3 \pmod{14}$. Thus $\Delta_b=2^b5^b2^b-2^b-5^b+1
\equiv 4\times 3\times 4-4-3+1\equiv 0\pmod {14}$.
Hence $\frac{\Delta_b}{\Delta_1}\in {\bf Z}$ and
$(S_1)|(S_1^b)$ in this case. Part (i) is proved.

(ii). Let $S_2=S_1$ with $S_1$ being given as in (\ref{eq 3.2}).
We can calculate and obtain that
\begin{small}
\begin{equation}\label{eq 3.6}
\begin{aligned}
&\ \ (S_2)^{-1}[S_2^b]\\
=&
\begin{pmatrix}
1+\frac{1}{u-1}+\frac{1}{v-1}+\frac{1}{\Delta_1}&\frac{1}{1-u}-\frac{1}{\Delta_1}&\frac{1}{1-v}-\frac{1}{\Delta_1}&\frac{1}{\Delta_1}\\
\\\frac{1}{1-u}-\frac{1}{\Delta_1}&\frac{1}{u-1}+\frac{1}{\Delta_1}&\frac{1}{\Delta_1}&-\frac{1}{\Delta_1}\\
\\
\frac{1}{1-v}-\frac{1}{\Delta_1}&\frac{1}{\Delta_1}&\frac{1}{v-1}+\frac{1}{\Delta_1}&-\frac{1}{\Delta_1}\\
\\\frac{1}{\Delta_1}&-\frac{1}{\Delta_1}&-\frac{1}{\Delta_1}&\frac{1}{\Delta_1}
\end{pmatrix}\cdot
\begin{pmatrix}
1&u^b&v^b&(uvw)^b\\
\\u^b&u^b&(uv)^b&(uvw)^b\\
\\v^b&(uv)^b&v^b&(uvw)^b\\
\\(uvw)^b&(uvw)^b&(uvw)^b&(uvw)^b
\end{pmatrix}\\
\\
=&\begin{pmatrix}
1+\frac{1-u^b}{u-1}+\frac{1-v^b}{v-1}+\frac{\Delta_b}{\Delta_1}&u^b+u^b\cdot
\frac{1-v^b}{v-1}+\frac{\Gamma_b}{\Delta_1} &v^b+v^b\cdot
\frac{1-u^b}{u-1}+\frac{\Gamma_b}{\Delta_1} &(uvw)^b\\
\\
\frac{1-u^b}{1-u}-\frac{\Delta_b}{\Delta_1}&-\frac{\Gamma_b}{\Delta_1}&v^b\cdot
\frac{1-u^b}{1-u}-\frac{\Gamma_b}{\Delta_1}&0\\
\\
\frac{1-v^b}{1-v}-\frac{\Delta_b}{\Delta_1}&u^b\cdot \frac{1-v^b}{1-v}-
\frac{\Gamma_b}{\Delta_1}&-\frac{\Gamma_b}{\Delta_1}&0\\
\\
\frac{\Delta_b}{\Delta_1}&\frac{\Gamma_b}
{\Delta_1}&\frac{\Gamma_b}{\Delta_1}&0
\end{pmatrix},
\end{aligned}
\end{equation}
\end{small}\\
where $\Delta_b$ is given as in (\ref{eq 3.5}) and
$\Gamma_b:=u^bv^b(w^b-1)$. By (\ref{eq 3.6}),
we know that the necessary and sufficient conditions for
$(S_2)^{-1}[S_2^b]\in M_4({\bf Z})$
are $\frac{\Delta_b}{\Delta_1}\in {\bf Z}$
and $\frac{\Gamma_b}{\Delta_1}\in {\bf Z}$.
We divide the following proof of part (ii) into two cases.

{\sc Case 2-1.}
Picking $u=2,v=3,w=2$ gives us that
$$
\frac{\Delta_b}{\Delta_1}=\frac{2^b3^b2^b-2^b-3^b+1}
{2\times 3\times 2-2-3+1}
=\frac{2^b(6^b-1)}{2^3}-\frac{3^b-1}{2^3}
$$
and
$$
\frac{\Gamma_b}{\Delta_1}=\frac{2^b3^b(2^b-1)}{2^3}.$$

Since $3^2\equiv 1\pmod 8$, we know that if $b$ is even and $b\ge 4$, then
$\frac{\Delta_b}{\Delta_1}\in {\bf Z}$ and $\frac{\Gamma_b}{\Delta_1}\in {\bf Z}$.
Hence $(S_2)|[S_2^b]$ in this case.

{\sc Case 2-2.}
Letting $u=3,v=4,w=2$, we have $\Delta_1=3\times 4\times 2-3-4+1=18$.
Let $b\equiv 3\pmod 6$. As the proof of case 2-3 of part (ii), we arrive at
$3^b\equiv 9\pmod{18}$, $4^b\equiv 10\pmod{18}$ and $2^b\equiv 8\pmod{18}$.
So
$$\Delta_b=3^b4^b2^b-3^b-4^b+1\equiv 9\times 10\times 8-9-10+1\equiv 0\pmod {18}$$
and
$$\Gamma_b=3^b4^b(2^b-1)\equiv 9\times 10\times (8-1)\equiv 0\pmod {18}.$$
Thus $\frac{\Delta_b}{\Delta_1}\in {\bf Z}$ and
$\frac{\Gamma_b}{\Delta_1}\in {\bf Z}$. So $(S_2)|[S_2^b]$ in this case.
Part (ii) is proved.

(iii). Let $S_3=\{1,3,5,45\}$. Then $S_3$ is a gcd-closed
set with $\max_{x\in S_3}\{|G_{S_3}(x)|\}=2$ and
$S_3$ does not satisfy the condition $\mathcal{G}$
since $G_{S_3}(45)=\{3, 5\}$ and $[3, 5]=15<45$.
We calculate and get that
\begin{align*}
&[S_3^5][S_3]^{-1}=
\begin{pmatrix}
1&3^5&5^5&45^5\\
3^5&3^5&15^5&45^5\\
5^5&15^5&5^5&45^5\\
45^5&45^5&45^5&45^5
\end{pmatrix}
\cdot
\begin{pmatrix}
1&3&5&45\\
3&3&15&45\\
5&15&5&45\\
45&45&45&45
\end{pmatrix}^{-1}\\
=&
\begin{pmatrix}
1&243&3125&184528125\\
243&243&759375&184528125\\
3125&759375&3125&184528125\\
184528125&184528125&184528125&184528125
\end{pmatrix}
\cdot
\begin{pmatrix}
\dfrac{13}{44}&-\dfrac{2}{11}&-\dfrac{7}{44}&\dfrac{1}{22}\\
\\
-\dfrac{2}{11}&\dfrac{2}{33}&\dfrac{3}{22}&-\dfrac{1}{66}\\
\\
-\dfrac{7}{44}&\dfrac{3}{22}&\dfrac{7}{220}&-\dfrac{1}{110}\\
\\
\dfrac{1}{22}&-\dfrac{1}{66}&-\dfrac{1}{110}&\dfrac{1}{990}
\end{pmatrix}\\
=&\begin{pmatrix}
8387101&-2795440&-1677396&186360\\
8266860&-2692359&-1653372&179496\\
8250000&-2750000&-1574375&175000\\
0&0&0&4100625
\end{pmatrix}\in M_4({\bf Z}).
\end{align*}
\noindent Hence $[S_3]|[S_3^5]$ holds in the ring
$M_4({\bf Z})$.

Let $S_3=\{1,2,3,4,24\}$. Then $S_3$ is gcd closed
and $\max_{x\in S_3}\{|G_{S_3}(x)|\}=2$. Since
$G_{S_3}(24)=\{3, 4\}$ and $[3, 4]=12<24$, the set
$S_3$ does not satisfy the condition $\mathcal{G}$.
But
\begin{align*}
&\ [S_3^{11}][S_3]^{-1}\\
=&
\begin{pmatrix}
1&2^{11}&3^{11}&4^{11}&24^{11}\\
2^{11}&2^{11}&6^{11}&4^{11}&24^{11}\\
3^{11}&6^{11}&3^{11}&12^{11}&24^{11}\\
4^{11}&4^{11}&12^{11}&4^{11}&24^{11}\\
24^{11}&24^{11}&24^{11}&24^{11}&24^{11}
\end{pmatrix}
\cdot
\begin{pmatrix}
1&2&3&4&24\\
2&2&6&4&24\\
3&6&3&12&24\\
4&4&12&4&24\\
24&24&24&24&24
\end{pmatrix}^{-1}\\
\\
=&
\begin{tiny}
\begin{pmatrix}
1&2048&177147&4194304&1521681143169024\\
2048&2048&362797056&4194304&1521681143169024\\
177147&362797056&177147&743008370688&1521681143169024\\
4194304&4194304&743008370688&4194304&1521681143169024\\
1521681143169024&1521681143169024&1521681143169024&1521681143169024&1521681143169024
\end{pmatrix}
\end{tiny}\\
\\
&
\times \begin{pmatrix}
-\frac{7}{22}&1&-\frac{5}{22}&-\frac{6}{11}&\frac{1}{11}\\
\\
1&-\frac{3}{2}&0&\frac{1}{2}&0\\
\\
-\frac{5}{22}&0&\frac{5}{66}&\frac{2}{11}&-\frac{1}{33}\\
\\
-\frac{6}{11}&\frac{1}{2}&\frac{2}{11}&-\frac{5}{44}&-\frac{1}{44}\\
\\
\frac{1}{11}&0&-\frac{1}{33}&-\frac{1}{44}&\frac{1}{264}
\end{pmatrix}\\
\\
=&
\begin{tiny}
\begin{pmatrix}
138334647052987&2094081&-46111549016980&-34583662788144&5763943623432\\
138334564638720&2096128&-46111521546240&-34583596858368&5763932635136\\
137929734786330&370960166907&-45976457388807&-34667913780036&5747057180982\\
138165784412160&0&-46055261470720&-34448570580992&5741428604928\\
0&0&0&0&63403380965376
\end{pmatrix}\in M_5({\bf Z}).
\end{tiny}
\end{align*}
Thus $[S_3]|[S_3^{11}]$ holds in the ring $M_5({\bf Z})$.
Part (iii) is proved.

This concludes the proof of Theorem \ref{t1.3}. \hfill$\Box$

\section{Final remarks}

Let $S$ be a gcd-closed set and let $a$ and $b$ be positive
integers such that $a|b$. If $\max_{x\in S}\{|G_S(x)|\}=1$,
then by Zhu's theorem \cite{Z-IJNT22} and the Zhu-Li theorem
\cite{ZL-BAMS22}, one knows that $(S^a)|(S^b), (S^a)|[S^b]$
and $[S^a]|[S^b]$ hold in the ring $M_n({\bf Z})$. From
Theorem {\rm\ref{t1.2}} of this paper we know that such
factorizations are true if $\max_{x\in S}\{|G_S(x)|\}=2$
and the set $S$ satisfies the condition $\mathcal G$.
When $a=b$, for any gcd-closed sets $S$ with $\max_{x\in S}
\{|G_S(x)|\}\ge 2$, it was conjectured in \cite{ZCH-JCTA22}
that such factorizations are true if and only if the set $S$
satisfies the condition $\mathcal{G}$. By Theorem {\rm\ref{t1.3}},
one knows the existences of positive integers $b>1$ and
gcd-closed sets $S$ with $\max_{x\in S}\{|G_S(x)|\}=2$
and the condition $\mathcal{G}$ not being satisfied, such
that $(S)|(S^b)$ (resp. $(S)|[S^b]$ and $[S]|[S^b]$) holds in
the ring $M_n({\bf Z})$. In other words, when $a|b$ and $a<b$,
the condition $\mathcal{G}$ is a sufficient and unnecessary
condition for the truth of Theorem \ref{t1.2}. However,
it is not clear that for each integer $b>1$, there is a
gcd-closed set $S$ with $\max_{x\in S}\{|G_S(x)|\}\ge 2$
and the condition $\mathcal{G}$ not being satisfied such
that $(S)|(S^b)$ (resp. $(S)|[S^b]$ and $[S]|[S^b]$) holds
in the ring $M_n({\bf Z})$. This problem is still kept open.\\


\end{document}